\documentclass{conm-p-l}

\usepackage{MnSymbol}
\usepackage{tikz-cd}
\usepackage{url}
\usepackage{comment}

\newtheorem{theorem}{Theorem}[section]
\newtheorem{proposition}[theorem]{Proposition}

\newtheorem{corollary}[theorem]{Corollary}

\theoremstyle{definition}
\newtheorem{definition}[theorem]{Definition}
\newtheorem{example}[theorem]{Example}

\theoremstyle{remark}

\numberwithin{equation}{section}

\newcommand{\mQ}{\mathbb{Q}}
\newcommand{\we}{\stackrel{\simeq}{\to}}
\newcommand{\fib}{\twoheadrightarrow}
\newcommand{\cofib}{\hookrightarrow}
\newcommand{\apl}{{\rm A_{PL}}}

\newcommand{\Ho}{{\rm H}}
\newcommand{\secat}{{\rm secat}}

\newcommand{\msecat}{{\rm msecat}}

\newcommand{\hsecat}{{\rm Hsecat}}

\newcommand{\cat}{{\rm cat}}
\newcommand{\nil}{{\rm nil}}

\newcommand{\im}{{\rm im}}
\newcommand{\hnil}{{\rm Hnil}}
\newcommand{\tc}{{\rm TC}}
\newcommand{\mtc}{{\rm mTC}}
\newcommand{\htc}{{\rm HTC}}
\newcommand{\mcat}{{\rm mcat}}

\newcommand{\id}{{\rm Id}}

\begin{document}

\title[Rational methods applied to sectional category and \tc]{Rational methods applied to sectional category and topological complexity}

\author{Jos\'e Carrasquel}
\address{Faculty of Mathematics and Computer Science, Adam Mic\-kie\-wicz University, Umultowska 87, 61-614 Pozna\'n, Poland.}
\email{jgcarras@amu.edu.pl, jgcarras@gmail.com}
\thanks{The author was supported in part by the Belgian Interuniversity Attraction Pole (IAP) within the framework ``Dynamics, Geometry and Statistical Physics" (DYGEST P7/18) and the Polish National Science Centre grant 2016/21/P/ST1/03460 within the European Union's Horizon 2020 research and innovation programme under the Marie Sk\l odowska-Curie grant agreement No. 665778.}


\subjclass[2010]{55M30, 55P62}


\keywords{rational homotopy, topological complexity, Lusternik-Schnirelmann category, sectional category}

\begin{abstract}
	This survey is a guide for the non specialist on how to use rational homotopy theory techniques to get approximations of Farber's topological complexity, in particular, and of Schwarz's sectional category, in general.
\end{abstract}

\maketitle

\tableofcontents

\section*{Introduction}
The \emph{sectional category} \cite{Schwarz66} of a continuous map $f\colon X\rightarrow Y$ is the least integer $m$ for which there are $m+1$ local homotopy sections for $f$ whose domains form an open cover of $Y$. If $X$ is a path-connected topological space then two important invariants of the homotopy type of $X$ can be described through sectional category. The first one is the \emph{Lusternik-Schnirelmann (LS) category} of $X$ \cite{Lusternik34}, $\cat(X)$, which is the sectional category of the base point inclusion \[\cat(X)=\secat(*\cofib X).\] The second one is the \emph{(higher) topological complexity} \cite{Farber03,Rudyak10} of $X$, $\tc_n(X)$, which is the sectional category of the $n$ diagonal inclusion \[\tc_n(X)=\secat(X\cofib X^{n}).\]

In this paper we will only consider, unless stated otherwise, simply connected CW complexes of finite type. As a consequence, every time we write the word \emph{space}, we actually mean one of such CW complexes.\\

If $X$ is a space, then there exists a cofibration $\rho_X\colon X\cofib X_0$, called the \emph{rationalization} of $X$, which verifies:

\begin{itemize}
	\item $X_0$ is a \emph{rational space}, that is, $\Ho^*(X,\mathbb{Z}
	)$ (or equivalently $\pi_*(X)$) is a rational vector space.
	\item $\rho_X$ is a \emph{weak rational homotopy equivalence}, that is, $\Ho^*(\rho_X, \mQ)$ (or equivalently $\pi_*(\rho_X)\otimes\mQ$)	is an isomorphism.
	\item Every continuous map $f\colon X\rightarrow Y$, with $Y$ a rational space, factors uniquely up to homotopy:
	\begin{center}
		\begin{tikzcd}
			X_0\arrow[dr, "f_0"]&\\
			X\arrow[u, "\rho_X", hookrightarrow]\arrow[r, "f"']&Y\\
		\end{tikzcd}
	\end{center}
\end{itemize}

The main lower bound for sectional category is cohomological. Namely, if $f$ is a continuous map and $R$ is any ring, then \[\nil\ker \Ho^*(f,R)\le \secat(f),\] where $\nil\ I$ denotes the \emph{nilpotency} of an ideal $I$, that is, the longest non trivial product of elements in $I$. Often, this bound is not good enough. This is mainly because cohomology does not capture all the homotopic information of a map. However, rational homotopy theory offers a whole new set of algebraic lower bounds which are based on the following inequalities: \[\nil\ker \Ho^*(f,\mQ)\le \secat(f_0)\le \secat(f).\]

More precisely, there is a pair of contravariant adjoint functors 
	\begin{center}
		\begin{tikzcd}
		\apl\colon \mathbf{Top}\arrow[r, shift left]&\mathbf{cdga}\arrow[l, shift left]\ :|\cdot|
		\end{tikzcd}
	\end{center}
where

\begin{itemize}
	\item \textbf{cdga} is the category of simply connected commutative differential graded algebras over $\mQ$ of finite type (see Section \ref{sec:SullivansRtc}).
	
	\item $\mathbf{Top}$ is the category of simply connected  CW complexes of finite type.
\end{itemize}

When restricted to rational spaces, these functors actually yield an equivalence in the associated homotopy categories \cite{Bousfield76,Su77}. This means that all the rational homotopic information of a continuous map $f$ can be encoded algebraically through $\apl(f)$. From this algebraic object, one can deduce approximations for sectional category which are better than the cohomological lower bound,  \[\nil\ \ker H^*(f,\mQ)\le \hsecat(f)\le \msecat(f)\le \secat(f_0)\le\secat(f),\] by relaxing the algebraic characterization of $\secat(f_0)$.

\section{Sullivan's rational homotopy theory}\label{sec:SullivansRtc}
For a deep description of the tools we use, the reader is invited to read the standard reference on rational homotopy theory \cite{Bible}. We start describing the algebraic objects we use.\\

A \emph{commutative graded algebra} (cga) is a graded $\mQ$-vector space \[A=\bigoplus_{i\ge 0}A^i\] together with a degree zero bilinear map $\mu\colon A\otimes A\to A$, $a\otimes b\mapsto ab$, called \emph{multiplication}, verifying for $a,b,c\in A$,
\begin{itemize}
\item $A^0=\mQ$ and $a1=1a=a$,
\item $(ab)c=a(bc)$, and
\item $ab=(-1)^{|a||b|}ba$.
\end{itemize}
where $|a|=i$ means that $a\in A^i$ and we say that $a$ is \emph{homogeneous} of degree $i$. Observe that, if $a$ has odd degree, then $a^2=0$. A \emph{morphism} of cga $f\colon A\to B$ is a linear map of degree zero such that $f(1)=1$, $f(ab)=f(a)f(b)$. A \emph{derivation} of degree $k$ on $A$ is a degree $k$ linear map $\theta\colon A\to A$ verifying $\theta(ab)=\theta(a)b+(-1)^{k|a|}a\theta(b)$. A \emph{differential} on $A$ is a derivation $d$ of degree $1$ such that $d^2=d\circ d=0$.\\

A \emph{commutative differential graded algebra} (cdga) is a pair $(A,d)$ where $A$ is a cga and $d$ is a differential on $A$. A \emph{morphism} of cdga is a morphism of cga commuting with differentials. We denote by \textbf{cdga} this category. There exists a \emph{homology} functor $\Ho\colon \mathbf{cdga}\to \mathbf{cdga}$ defined as $\Ho(A,d):=\left(\frac{\ker d}{d(A)},0\right)$ and $\Ho(f)([z])=[f(z)]$. A morphism of cdga $f$ is said to be a \emph{quasi-isomorphism} if $\Ho(f)$ is an isomorphism.\\

The \emph{direct sum} of two cdgas is a cdga, $(A,d)\oplus (B,d):=(A\oplus_\mQ B,d)$, with $d(a+b)=d(a)+d(b)$ and $ab=0$ for $a\in A$ and $b\in B$.\\

The \emph{tensor product} of two cdgas is a cdga, $(A,d)\otimes (B,d):=(A\otimes_\mQ B,d)$, with $(a\otimes b)(a'\otimes b'):=(-1)^{|b||a'|}(aa')\otimes(bb')$ and $d(a\otimes b):= d(a)\otimes b + (-1)^{|a|}a\otimes d(b)$.\\

Given a graded $\mQ$-vector space $V$, define $TV=\bigoplus_{j\ge 0}T^jV$ where $T^0V=\mQ$ and $T^jV=V^{\otimes j}$. We have that $TV$ is a graded algebra with product \[(a_1\otimes\cdots \otimes a_j)(b_1\otimes\cdots \otimes b_k):=a_1\otimes\cdots \otimes a_j\otimes b_1\otimes\cdots \otimes b_k.\] Consider $I$ the ideal of $TV$ generated by elements of the form $a\otimes b - (-1)^{|a||b|} b\otimes a$ and define the \emph{free commutative graded algebra} generated by $V$: \[\Lambda V:=\frac{TV}{I}.\] We will write simply $ab$ to denote the element $[a\otimes b]\in\Lambda V$. We say that an element $v_1\cdots v_j\in\Lambda^jV$ has \emph{word length} $j$ and degree $|v_1|+\cdots +|v_j|$. We will denote by $A^+$ the elements of positive degree of a graded algebra $A$, and by $\Lambda^+ V$ the elements of positive word length of a free graded algebra. We will also employ notations such as $\Lambda^{>m}V$ to mean $\bigoplus_{j>m}\Lambda^j V$, and so on. Observe that there is an isomorphism $\Lambda V\otimes\Lambda W\cong\Lambda(V\oplus W)$ which we often will use implicitly.\\

The free commutative graded algebra $\Lambda V$ has the following two universal properties:\label{universalProps}

\begin{itemize}
	\item Given $A$ a cga and $f\colon V\to A$ a linear map of degree zero, then there exists an unique cga morphism $\hat{f}\colon \Lambda V\to A$ verifying $\hat{f}(v)=f(v)$ for all $v\in V$.
	\item If $\theta\colon V\to\Lambda V$ is a linear map of degree $k$ then there exists a unique derivation on $\Lambda V$ of degree $k$, $\hat{\theta}$, such that $\theta(v)=\hat{\theta}(v)$ for all $v\in V$. 
\end{itemize}

Let $(A,d)$ be a cdga. A \emph{relative Sullivan algebra} over $(A,d)$ is a cdga of the form $(A\otimes\Lambda V, D)$ such that $D(a)=d(a)$, if $a\in A$, and $V=\bigoplus_{k\ge 1} V(k)$ with $D(V(1))\subset A\otimes 1$ and \[D(V(k))\subset A\otimes \Lambda\left(V(1)\oplus\cdots\oplus V(k-1)\right) \mbox{ (\emph{nilpotence condition}).}\] If, in addition, the differential $D$ satisfies $D(V)\subset A^+\otimes\Lambda V+A\otimes\Lambda^{\ge 2}V$ we say that $(A\otimes \Lambda V, D)$ is a \emph{minimal relative Sullivan algebra}. Observe that the inclusion $i\colon (A,d)\rightarrow (A\otimes\Lambda V,D)$ is a cdga morphism. In fact, the category \textbf{cdga} is a closed model category \cite{Quillen67} (if we restrict to cdgas $(A,d)$ for which $\Ho^0(A,d)=\mQ$ \cite{Ha83}) with the following structure:

\begin{itemize}
	\item Weak equivalences: quasi-isomorphisms $(A,d)\we(B,d)$.
	\item Fibrations: surjective morphisms $(A,d)\fib(B,d)$.
	\item Cofibrations: Inclusions of a cdga $(A,d)$ into a relative Sullivan algebra \[(A,d)\cofib (A\otimes\Lambda V,D).\]
\end{itemize}

As a consequence we get the following important fact:\\

\textbf{Existence of relative Sullivan models:} Every cdga morphism $\varphi$ can be factored as
\begin{center}
\begin{tikzcd}
	(A,d)\arrow[rr, "\varphi"]\arrow[dr, hook, "i"']&&(B,d)\\
	&(A\otimes\Lambda V,D),\arrow[ur, "\theta"', "\simeq" ]
\end{tikzcd}
\end{center}
where either $\theta$ is surjective or $(A\otimes\Lambda V, D)$ is minimal (but not necessarily both at the same time).\\

An \emph{$(A,d)$-module homotopy retraction} for $\varphi$ is a chain map \[r\colon (A\otimes \Lambda V, D)\rightarrow (A,d)\] such that $r(a)=a$ and $r(a\xi)=ar(\xi)$ for all $a\in A$ and $\xi\in A\otimes\Lambda V$. Moreover, if $r$ is a cdga morphism, we say that it is a \emph{(cdga) homotopy retraction} for $\varphi$.\\

\textbf{The surjective trick:} Every cdga morphism $\varphi$ can be factored as
\begin{center}
	\begin{tikzcd}
		(A,d)\arrow[rr, "\varphi"]\arrow[dr, hook, "\simeq"']&&(B,d)\\
		&(A\otimes\Lambda V,D),\arrow[ur, twoheadrightarrow ]
	\end{tikzcd}
\end{center}

\textbf{The lifting lemma:} For any solid commutative cdga diagram, there exists a dashed arrow completing the commutative cdga diagram

\begin{center}
	\begin{tikzcd}
		(A,d)\arrow[r]\arrow[d, hook]&(B,d)\arrow[d, twoheadrightarrow ]\\
		(A\otimes\Lambda V,D)\arrow[r]\arrow[ur, dashed]&(C,d),
	\end{tikzcd}
\end{center}
provided that at least one of the vertical morphisms is a quasi-isomorphism.\\

Since \textbf{cdga} is a closed model category, there exists a notion of homotopy of maps $(\Lambda V, d)\rightarrow (B,d)$ which is nicely described in \cite[Chap. 12-14]{Bible}. Two cdga morphism $\varphi_1,\varphi_2$ are \emph{weakly equivalent} when there is a homotopy commutative cdga diagram

\begin{equation*}
\begin{tikzcd}
(A_1,d)\ar[d, "\varphi_1"']&(\Lambda V,d)\ar[d]\ar[l, "\simeq"']\ar[r, "\simeq"]&(A_2,d)\ar[d, "\varphi_2"]\\
(B_1,d)&(\Lambda W,d)\ar[l, "\simeq"']\ar[r, "\simeq"]&(B_2,d),
\end{tikzcd}
\end{equation*}
where $(\Lambda V,d)$ and $(\Lambda W,d)$ are Sullivan algebras (see below).

\subsection{Sullivan models}\label{ssec:SullivanMods}
In the special case that $(A,d)=(\mQ,0)$, the initial object of \textbf{cdga}, we get \emph{Sullivan} (perhaps \emph{minimal}) \emph{models} for a cdga $(B,d)$, $\theta\colon (\Lambda V, d)\we (B,d)$. These objects are very important as they are the fibrant-cofibrant objects in the category \textbf{cdga}. For us, unless stated otherwise, cdgas denoted in the form $(\Lambda V,d)$ will be assumed to be Sullivan algebras, the same applies to relative Sullivan algebras $(A\otimes\Lambda V,D)$.\\

To be more explicit, a \emph{Sullivan algebra} is a cdga of the form $(\Lambda V, d)$ verifying the \emph{nilpotence condition}, that is, $V=\bigoplus_{k\ge 1} V(k)$ with $D(V(1))=0$ and \[D(V(k))\subset \Lambda\left(V(1)\oplus\cdots\oplus V(k-1)\right).\] If moreover, the differential $D$ verifies $D(V)\subset \Lambda^{\ge 2}V$, we say that it is a \emph{minimal Sullivan algebra}. As with relative Sullivan models for cdga morphism, the construction of Sullivan models for a given cdga, $\theta\colon (\Lambda V,d)\we(B,d)$, can be carried out inductively, degree by degree, by adding generators to $V$ and defining their differentials to \emph{kill} or create the necessary homology classes in order to turn $\theta$ into a quasi-isomorphism. This process is very well explained in \cite[Pg. 144]{Bible} when $\Ho^1(B,d)=0$.\\

Let $(\Lambda V,d)$ be a Sullivan algebra and write $d=d_0+d_1+\cdots$ with $d_k(V)\subset \Lambda^{k+1}V$. The equation $d^2=0$ implies that $\sum_{i=0}^kd_id_{k-i}=0$ for all $k\ge 0$.  Observe then that $d_0$ is a differential on $V$, called the \emph{linear part} of $d$, which gives a chain complex $(V,d_0)$. Analogously, given a cdga morphism $\varphi\colon (\Lambda V,d)\rightarrow (\Lambda W,d)$ define the \emph{linear part} of $\varphi$ as the chain complex morphism \[Q(\varphi)\colon (V,d_0)\rightarrow (W,d_0)\] defined by $(\varphi-Q(\varphi))(V)\subset \Lambda^{\ge 2}W$.\\

We can now state some important facts about Sullivan algebras which will be used throughout this survey\cite[Prop. 14.13, Thm. 14.12]{Bible}.

\begin{proposition}\label{prop:SullivanAlgs} Let $\varphi\colon (\Lambda V,d)\rightarrow (\Lambda W,d)$ be a cdga morphism between Sullivan algebras.
	\begin{itemize}
		\item Then $\varphi$ is a quasi-isomorphism if and only if $\Ho(Q(\varphi))$ is an isomorphism.
		\item If $(\Lambda V, d)$ and $(\Lambda W,d)$ are minimal, then $\varphi$ is a quasi-isomorphism if and only if $\varphi$ is an isomorphism.
	\end{itemize} 
\end{proposition}

\subsection{The connection with topology}
Given a space $X$, we say that a cdga $(A,d)$ is a \emph{model} for $X$ if $\apl(X)$ is weakly equivalent to $(A,d)$, that is, if there is a Sullivan algebra $(\Lambda V,d)$ and quasi-isomorphisms

\begin{center}
	\begin{tikzcd}
		(A,d)&(\Lambda V,d)\arrow[l, "\simeq"'] \arrow[r, "\simeq"] & \apl(X).
	\end{tikzcd}
\end{center}
In this case, we say that $(\Lambda V,d)$ is a \emph{Sullivan model} of $X$ (and of $(A,d)$). These special \emph{cofibrant} models are suitable for encoding the rational homotopic information of $X$. In fact, by Proposition \ref{prop:SullivanAlgs}, two minimal models for a space $X$ are isomorphic. In fact, we have \cite[Cor. 10.10, Thm. 15.11]{Bible}

\begin{theorem}\label{th:connectinTopol}
	If $(\Lambda V,d)$ is \emph{the} minimal Sullivan model for $X$, then 

\begin{itemize}
	\item $\Ho^*(X,\mQ)\cong \Ho(\apl(X))\cong \Ho(\Lambda V,d)$, and
	\item for every $i\ge 1$ there is a natural linear isomorphism $V^i\rightarrow \hom_\mathbb{Z}(\pi_i(X),\mQ)$.
\end{itemize}
\end{theorem}
A cdga morphism $\varphi$ is said to be a \emph{model} for a continuous map $f$ if $\varphi$ is weakly equivalent to $\apl(f)$.\\

We will often use a special notation for describing Sullivan algebras which can be understood by the following example: $(\Lambda(a_i,b_j,c_k); db=\alpha, dc=\gamma)$ means $(\Lambda V,d)$ with $V$ spanned by $a,b,c$ of degrees $i,j,k$ respectively and with $da=0$, $db=\alpha\in\Lambda(a)$ and $dc=\gamma\in \Lambda(a,b)$. The universal properties of p. \pageref{universalProps} show that the previous $(\Lambda V,d)$ is well defined.

\begin{example}
Let us compute the minimal Sullivan models for spheres. We have to construct a quasi isomorphism $\theta\colon (\Lambda V,d)\we \apl(S^n)$. Recall that \[\Ho^*(\apl(S^n))=\mQ\left\langle  1,\Omega \right\rangle\] with $\Omega$ the fundamental class of $S^n$. So $V$ must have one generator of degree $n$ which is a cycle, say $a$. Take $\omega\in \apl(S^n)$ a cycle in degree $n$ representing $\Omega$ and define
\begin{center}
	\begin{tikzcd}[row sep=tiny]
		(\Lambda (a),0)\arrow[r, "\theta"]&\apl(S^n)\\
		a\arrow[r, mapsto]&\omega\\
	\end{tikzcd}
\end{center}

Here the dimension is crucial. If $n$ is odd, then $a^2=0$ and $\Ho(\theta)$ is an isomorphism. Therefore $(\Lambda (a),0)$ is the minimal model of an odd sphere. However, if  $n$ is even, then $\Ho(\Lambda (a), 0)=\mQ\langle 1, a, a^2, a^3,\ldots\rangle$ and $\theta$ cannot be a quasi-isomorphism, this means that we need to add new generators in $V$ that turn $a^2, a^3,\ldots$ into boundaries. So let $x$ be a new generator of degree $2n-1$ and define $d(x)=a^2$. Then in $(\Lambda(a,x),d)$ we have

\begin{center}
	\begin{tikzcd}[row sep=tiny]
		1&a&a^2&a^3&a^4\cdots \\
		&x\arrow[ur, "d"']&ax\arrow[ur, "d"']&a^2x\arrow[ur, "d"']&a^3x\cdots
	\end{tikzcd}
\end{center}
which shows that $\Ho(\Lambda(a,x),d)\cong \Ho(\apl(S^n))$. However, we still have to define the quasi-isomorphism $\theta$ to make sure that $(\Lambda(a,x),d)$ is the minimal model of an even sphere. Observe that, in $\Ho^*(\apl(S^n))$, $\Omega^2=0$, this means that there exists $\xi\in\apl(S^n)$ of degree $2n-1$ such that $d(\xi)=\omega^2$. We then define $\theta$ as

\begin{center}
	\begin{tikzcd}[row sep=tiny]
		(\Lambda (a,x),d)\arrow[r, "\theta"]&\apl(S^n)\\
		a\arrow[r, mapsto]&\omega\\
		x\arrow[r, mapsto]&\xi.
	\end{tikzcd}
\end{center}
Observe finally that $d(\theta(x))=d(\xi)=\omega^2=\theta(a)^2=\theta(a^2)=\theta(d(x))$. This proves that $\theta$ is a quasi-isomorphism showing that $(\Lambda(a,x),d)$ is the minimal Sullivan model of an even sphere. Notice that, as corollary, we get the Serre finiteness theorem for the homotopy groups of spheres.
\end{example}

In the previous example, there exists also a quasi-isomorphism \[(\Lambda V,d)\we (\Ho(\Lambda V, d),0)\] defined as $a\mapsto [a]$ and $x\mapsto 0$. When this happens we say that the space is \emph{formal}. More precisely, a cdga $(A,d)$ is said to be \emph{formal} when $(A,d)$ and $(\Ho(A),0)$ are weakly equivalent. A space $X$ is said to be \emph{formal} when $\apl(X)$ is a formal cdga.

\begin{example}\label{examp:NotFormalSpace}
	A cdga which is not formal. Consider \[(\Lambda V,d)=(\Lambda(a_3,b_3,x_5), dx=ab)\] and observe that, as vector spaces, \[\Lambda V=\mQ \langle 1,a,b,x,ab,ax,bx,abx \rangle\mbox{ and}\] \[\Ho(\Lambda V,d)=\mQ\langle 1, [a], [b], [ax], [bx], [abx]\rangle.\]  For degree reasons, any cdga morphism $\varphi \colon (\Lambda V,d)\rightarrow (\Ho(\Lambda V,d),0)$ must satisfy $\varphi(x)=0$, but then, we would have $\Ho(\varphi)([ax])=[\varphi(ax)]=[\varphi(a)0]=0$, so $\varphi$ cannot be a cdga quasi-isomorphism. This proves that $(\Lambda V,d)$ is not \emph{the} minimal model of the cdga $(\Ho(\Lambda V, d),0)$ and therefore $(\Lambda V,d)$ is not formal.
\end{example}

A cdga morphism $\varphi$ is said to be \emph{formal} when it is weakly equivalent to $\Ho(\varphi)$. A continuous map $f$ is said to be \emph{formal} if $\apl(f)$ is formal \cite{Vig79,Oprea86,FT88} or equivalently, when $\Ho^*(f,\mQ)$ is a model for $f$. Obviously, if a map $f\colon X\rightarrow Y$ is formal then both $X$ and $Y$ are formal spaces, however, the converse is not true as we will see in Example \ref{example:FibreHopfFib}.

\subsection{Models for homotopy pullbacks}\label{sect:ModelsForHpb}
One of the most frequent ways in rational homotopy theory to construct models is through homotopy pullbacks and pushouts. Suppose that either $f$ or $g$ is a fibration in the following commutative diagram

\begin{equation}\label{diag:CommSquare}
	\begin{tikzcd}
		A\ar[d, "\alpha"']\ar[r, "\beta"]&X\ar[d, "f"]\\
		Y\ar[r, "g"']&Z.
	\end{tikzcd}
\end{equation}
Then the universal property of pullbacks induces a diagram

\begin{equation}\label{diag:pullbackwhisker}
\begin{tikzcd}
	A\ar[drr, bend left, "\beta"]\ar[ddr, bend right, "\alpha"']\ar[dr, "{(\alpha,\beta)}" description]&&\\
	&T\ar[d, "\hat{f}"']\ar[r, "\hat{g}"]&X\ar[d, "f"]\\
	&Y\ar[r, "g"']&Z.
\end{tikzcd}
\end{equation}
in which the square is a (homotopy) pullback\cite{Ma76}. We now explain how to construct cdga models for the previous diagram. Suppose we have a commutative cdga square modeling Diagram \ref{diag:CommSquare} (applying $\apl$ for example),

\begin{equation*}
\begin{tikzcd}
(A,d)\ar[d, "g"']\ar[r, "f"]&(B,d)\ar[d, "\beta"]\\
(C,d)\ar[r, "\alpha"']&(D,d),
\end{tikzcd}
\end{equation*}
where the name of a cdga morhpism is the same as the name of the continuous maps it models (for instance, $f\colon (A,d)\rightarrow (B,d)$ is a cdga model for $f\colon X\rightarrow Z$). We will often adopt this convention. Now choose $f$ or $g$ and factor it as a cofibration followed by a weak equivalence, say $f=\theta\circ i$ with $i\colon (A,d)\cofib (A\otimes\Lambda V, D)$ and $\theta\colon (A\otimes\Lambda V, D)\we (B,d)$. Since $C\otimes_A(A\otimes\Lambda V)\cong C\otimes\Lambda V$, diagram \ref{diag:pullbackwhisker} is modeled by

\begin{equation*}
\begin{tikzcd}
&(B,d)\ar[dddr, bend left, "\beta"]\\
(A,d)\arrow[ur, "f"]\arrow[r, hookrightarrow, "i"']\arrow[d, "g"']& (A\otimes\Lambda V,D)\arrow[u, "\simeq", "\theta"']\arrow[d, "\hat{g}"]\\
(C,d)\ar[drr, bend right, "\alpha"']\arrow[r, hookrightarrow, "\hat{f}"']&(C\otimes\Lambda V,\overline{D})\ar[dr, "{(\alpha,\beta)}" description]\\
&&(D,d)
\end{tikzcd}
\end{equation*}
in which the square is a pushout in \textbf{cdga}, $\hat{g}(a)=g(a)$, $g(v)=v$, $\overline{D}(v)=\hat{g}(Dv)$ for $a\in A$ and $v\in V$. Observe also that the induced morphism $(\alpha,\beta)$ is given by $(\alpha,\beta)(c)=\alpha(c)$, $(\alpha,\beta)(v)=\beta(\theta(v))$.\\

\subsection{Models for fibrations}

Let $p\colon E\fib B$ be a fibration with fiber $F$. Then we have a (homotopy) pullback

\begin{equation*}
\begin{tikzcd}
F\ar[d]\ar[r]\ar[d]&E\ar[d, twoheadrightarrow, "p"]\\
*\ar[r, hook]&B.
\end{tikzcd}
\end{equation*}
Following previous section, take $i\colon(A,d)\cofib (A\otimes \Lambda V,D)$ a relative Sullivan model for $p$. Then we have a model for the previous diagram

\begin{equation*}
\begin{tikzcd}
(A,d)\ar[r, hook, "i"]\ar[d, twoheadrightarrow]&(A\otimes \Lambda V, D)\ar[d, twoheadrightarrow, "q"]\\
(\mQ,0)\ar[r, hook] & (\Lambda V, \overline{D}),
\end{tikzcd}
\end{equation*}
where the projection \[q\colon (A\otimes\Lambda V,D)\fib \left(\frac{A\otimes\Lambda V}{A^+\otimes\Lambda V}, \overline{D}\right)\cong (\Lambda V, \overline{D})\] is a model for the inclusion $F\cofib E$\cite[Theorem 15.3]{Bible}.

\begin{example}[The Hopf fibration]\label{example:FibreHopfFib} Consider the Hopf fibration $p\colon S^7\rightarrow S^4$. A cdga model for $p$ is the projection $q\colon(\Lambda (a_4,x_7),dx=a^2)\fib (\Lambda (x),0)$, since there is a bijection $[S^7_0,S^4_0]\cong [(\Lambda(a,x),d),(\Lambda(x),0)]$. Construct now a relative Sullivan model for $q$ (and thus, for $p$) as follows: Since $\Ho(q)([a])=0$, introduce a generator $y$ of degree $3$ and define $Dy=a$. Now extend $q$ to \[\theta\colon (\Lambda(a,x)\otimes \Lambda (y),D)\rightarrow (\Lambda(x),0)\] by setting $\theta(y)=0$. Now observe that $\Ho((\Lambda(a,x)\otimes \Lambda (y),D))=\mQ\langle 1, [x-ay]\rangle$, therefore, $\theta$ is a quasi-isomorphism (Proposition \ref{prop:SullivanAlgs}). This gives a model of the fiber of $f$ \[\left(\frac{\Lambda(a,x)\otimes \Lambda (y)}{\Lambda^+(a,x)\otimes \Lambda (y)},\overline{D} \right)\cong(\Lambda(y),0),\] which is $S^3$. Observe lastly that $\Ho(q)$ is trivial but $p$ is not, this shows that $p$ is a map between formal spaces which is not formal.
\end{example}

\subsection{Models for homotopy pushouts}
Dually, one can also model homotopy pushouts throught pullbacks provided one of the cdga models is surjective. Consider 
\begin{equation*}
\begin{tikzcd}
&(Q,d)\ar[d, twoheadrightarrow]\ar[r]&(A,d)\ar[d, twoheadrightarrow, "f"]\\
&(B,d)\ar[r, "g"']&(X,d),
\end{tikzcd}
\end{equation*}
where $(Q,d)$ is the sub cdga of the direct sum $(A,d)\oplus(B,d)$, \[Q=\left\{(a,b)\in A\oplus_\mQ B\colon f(a)=g(b)\right\}.\]

\subsection{Models for cofibrations} If $i\colon C\cofib X$ is a cofibration modeled by a surjective cdga morphism $\varphi\colon (A,d)\rightarrow (B,d)$, the projection $X\rightarrow X/C$ is modeled by the the inclusion \[(\mQ\oplus\ker\varphi,d)\rightarrow (A,d).\] The weakness of this approach, in contrast with the models of fibration, is that the model for the cofiber, $\mQ\oplus\ker\varphi$, need not be a Sullivan model. 

\begin{example}[The homotopy cofibre of the Hopf fibration]
Recall $q$ the model of the Hopf fibration $p$ in Example \ref{example:FibreHopfFib}. Then a model for the homotopy cofiber of $p$ is given by $\mQ\oplus \ker q=\mQ\oplus (a)$. One can see that the Sullivan model for this space is given by $(\Lambda (a_4,y_{11}), dy=a^3)$.
\end{example}

\subsection{Models for the base point inclusion}
Let $(A,d)$ be a cdga model for a space $X$ with $A^0=\mQ$, then the \emph{augmentation} morphism $\epsilon\colon (A,d)\to (\mQ,0)$ is a cdga model for the base point inclusion $*\cofib X$. Observe that this inclusion is homotopy equivalent to the based path space fibration $e\colon PX\rightarrow X$, $\alpha\mapsto\alpha(1)$ whose fiber is the loop space $\Omega X$. Now let $\theta\colon(\Lambda V,d)\we(A,d)$ be the minimal Sullivan model for $(A,d)$ and take a relative Sullivan model for the augmentation $(\Lambda V,d)\rightarrow (\mQ,0)$ which is of the form $(\Lambda V\otimes\Lambda\hat{V},D)$ with $D_0\colon \hat{V}\rightarrow V$ a degree 1 linear isomorphism. Moreover, since the fiber of $e$, $\Omega X$, is an H-space, by \cite[Pg. 143]{Bible} we must have that $D(\hat{V})\subset \Lambda^+V\otimes\Lambda\hat{V}$. Now the following homotopy pushout diagrams gives a relative model for $\epsilon$
\begin{equation*}
\begin{tikzcd}
(\Lambda V,d)\ar[r, hook]\ar[d, "\theta"', "\simeq"]&(\Lambda V\otimes\Lambda\hat{V},D)\ar[d, "\theta\otimes\id"',"\simeq"]\ar[ddr, bend left, "\simeq"]\\
(A,d\ar[drr, bend right, "\epsilon"'])\ar[r, hook]&(A\otimes\Lambda\hat{V},\overline{D})\ar[dr, "\simeq"]&\\
&&(\mQ,0),
\end{tikzcd}
\end{equation*}
where $D(\hat{v})=(\theta\otimes\id)(D\hat{v})\in A^+\otimes\Lambda\hat{V}=(\ker\epsilon)\otimes\Lambda\hat{V}.$ This construction is called the \emph{acyclic closure} of $(A,d)$\cite[p. 192]{Bible}.

\subsection{Models for the diagonal map}\label{sec:modelForDiagonal}
Let $X$ be a space and $\Delta_X^n\colon X\cofib X^n$ the $n$ diagonal map $x\mapsto (x,x,\ldots,x)$. If $(A,d)$ is a cdga model for $X$, then the $n$-multiplication morphism $\mu_n\colon (A^{\otimes n},d)\fib A$ is a surjective cdga model for $\Delta_X^n$\cite[Pg. 142]{Bible}.\\

Let us now give a model of $\frac{X^n}{\Delta_X^n(X)}$, the cofibre of $\Delta^n_X$. If $a\in A$, denote by $a_i\in A^{\otimes n}$ the corresponding inclusion of $a$ into the $i$-th factor. Let $G$ be a set of generators for $A$, that is, $A^+=(G)$, the ideal of $A$ generated by $G$. By \cite[Section 6]{Carrasquel10}, $\ker \mu_n$ is generated by $\left\{x_1-x_i\colon x\in G, i=2,\ldots, n\right\}$. We see then that a cdga model for $\frac{X^n}{\Delta_X^n(X)}$ is $\mQ\oplus \left(\left\{x_1-x_i\colon x\in G, i=2,\ldots, n\right\}\right)$.\\

The diagonal inclusion $\Delta_X^n$ is homotopy equivalent to the path fibration $\pi_n\colon X^I \fib X^n$, $\alpha\mapsto \alpha\left(0,\frac{1}{n-1},\frac{2}{n-1}\ldots,\frac{n-2}{n-1},1\right)$ whose fiber is the product of based loops on $X$, $(\Omega X)^{n-1}$.\\

Let us now construct a cdga model for this fibration. Let $(\Lambda V,d)$ be a Sullivan model for $X$, then by \cite[Pg. 206]{Bible} a relative Sullivan model for the multiplication \[\mu_2\colon(\Lambda V_1\otimes\Lambda V_2,d)\rightarrow(\Lambda V,d),\] where $V\cong V_1\cong V_2,$ is given by $(\Lambda V_1\otimes\Lambda V_2\otimes \Lambda \hat{V},D)$, together with $\hat{V}=sV$, $\hat{V}^k=V^{k+1}$ and \[D\hat{v}=v_2-v_1-\sum_{k\ge 1}\frac{(sD)^k}{k!}(v_1).\] The map $s$ is the degree $-1$ derivation defined by $s(v_1)=s(v_2)=\hat{v}$, $s(v)=0$. Here, subscripts are used to distinguish copies of $V$. Observe that $D(\hat{V})\subset(\ker\mu_2)\oplus (\Lambda V_1\otimes\Lambda V_2\otimes \Lambda^+\hat{V})$ therefore a quasi-isomorphism \[\theta\colon (\Lambda V_1\otimes\Lambda V_2\otimes \Lambda \hat{V},D)\we(\Lambda V, d)\] is defined by $\theta(v_1)=\theta(v_2)=v$ and $\theta(\hat{v})=0$. Now $\pi_3$ fits in the homotopy pullback
\begin{equation*}
\begin{tikzcd}[column sep=huge]
X^I\ar[r]\ar[d, "\pi_3"']&X^I\times X^I\ar[d, "\pi_2\times\pi_2"]\\
X\times X\times X\ar[r, "\id\times\Delta_2\times\id"']&X\times X\times X\times X
\end{tikzcd}
\end{equation*}
which is modeled through the pushout
\begin{equation*}
\begin{tikzcd}
&(\Lambda V,d)\otimes(\Lambda V,d)\\
(\Lambda V\otimes\Lambda V,d)\otimes(\Lambda V\otimes\Lambda V,d)\ar[ur, "\mu_2\otimes\mu_2"]\ar[d, "\id\otimes\mu_2\otimes\id"']\ar[r, hook]&(\Lambda V\otimes\Lambda V\otimes\Lambda\hat{V},D)\otimes(\Lambda V\otimes\Lambda V\otimes\Lambda\hat{V},D)\ar[d]\ar[u, "\theta\otimes\theta"', "\simeq" ]\\
(\Lambda V_1\otimes\Lambda V_2\otimes\Lambda V_3,d)\ar[r, hook]&(\Lambda V_1\otimes\Lambda V_2\otimes\Lambda V_3\otimes\Lambda \hat{V}_1\otimes\Lambda\hat{V_2},D).
\end{tikzcd}
\end{equation*}
An inductive argument proves that the relative Sullivan model for $\mu_n$ is given by the inclusion \[(\Lambda V_1\otimes\cdots\otimes\Lambda V_n,d)\cofib(\Lambda V_1\otimes\cdots\otimes\Lambda V_n\otimes\Lambda( \hat{V}_1\oplus \cdots\oplus \hat{V}_{n-1}),D)\] where \[D\hat{v}_i=v_{i+1}-v_i-\sum_{k\ge 1}\frac{(s_iD)^k}{k!}(v_i),\ s_i(v_i)=s_i(v_{i+1})=\hat{v_i}.\]
Observe, in particular that the minimal model for the homotopy fibre of $\pi_n$ is given by \[(\Lambda sV,0)^{\otimes n-1}\cong(\Lambda sV^{\oplus n-1},0).\] This corresponds to the model given in \cite[Pg. 143]{Bible} for H-spaces and the fact that $\pi_{i-1}(\Omega X)\cong \pi_i(X)$.\\

\begin{example} Even spheres $S^n$. Let $(\Lambda(a,x), dx=a^2)$ be the minimal model for $S^n$. Then, with previous notation, the multiplication morphism is given by \[\mu_n\colon (\Lambda(a_1,\ldots, a_n,x_1,\ldots, x_n),d)\rightarrow (\Lambda(a,x),d),\]
with $\mu_n(a_i)=a$, $\mu_n(x_i)=x$, and $dx_i=a_i^2$, $i=1,\ldots, n$. Therefore a relative Sullivan model for $\mu_m$ is given by \[(\Lambda(a_1,\ldots, a_n,x_1,\ldots, x_n)\otimes\Lambda(\hat{a}_1,\ldots, \hat{a}_{n-1},\hat{x}_1,\ldots \hat{x}_{n-1}),D),\] with $|a_i|=n-1$ and $|x_i|=2n-2$. Let us now compute the differentials. Since $da_i=0$ we have $D\hat{a}_i=a_{i+1}-a_i$. We now compute $(s_iD)^k(x_i)$ using the fact that $D$ and $s_i$ are derivations:
\begin{itemize}
	\item $k=1$: $(s_iD)(x_i)=s_i(a_i^2)=2a_is(a_i)=2a_i\hat{a}_i$,
	\item $k=2$: $(s_iD)^2(x_i)=2(s_iD)(a_i\hat{a}_i)=2s_i(a_i(a_{i+1}-a_i))=2\hat{a}_i(a_{i+1}-a_i)$,
	\item $k=3$: $(s_iD)^3(x_i)=2(s_iD)(\hat{a}_i(a_{i+1}-a_i))=2s_i(a_{i+1}-a_i)^2=0.$
\end{itemize}
Therefore \[D\hat{x_i}=x_{i+1}-x_i-(2a_i\hat{a}_i+\hat{a}_i(a_{i+1}-a_i))=x_{i+1}-x_i-\hat{a}_i(a_{i+1}+a_i).\] Observe that $D^2=0$ and that the morphism \[\theta\colon ((\Lambda(a_1,\ldots, a_n,x_1,\ldots, x_n)\otimes\Lambda(\hat{a}_1,\ldots, \hat{a}_{n-1},\hat{x}_1,\ldots \hat{x}_{n-1}),D))\rightarrow (\Lambda(a,x),d)\] $\theta(a_i)=a$, $\theta(x_i)=x$, $\theta(\hat{a}_i)=\theta(\hat{x}_i)=0$ is a cdga quasi-isomorphism since $\Ho(Q(\theta))$ is an isomorphism.
	
\end{example}
\section{Rational Lusternik-Schnirelmann category}
We present this section as an overview of the rational homotopy techniques used to study LS category. The rational methods used to study sectional category are generalizations of the ones used for LS category.\\

We begin stating the F\'elix-Halperin theorem \cite[Theorem VIII]{Felix82} which gives an algebraic characterization of the LS category of rational spaces. This theorem is key in the development of computation methods for LS category.

\begin{theorem}\label{th:FelixHalperinCharac}
	Let $X$ be a space and $(\Lambda V,d)$ a Sullivan model for $X$. Then the rational LS category of $X$, $\cat(X_0)$, is the least $m$ for which the cdga projection \[\rho_m\colon (\Lambda V,d)\rightarrow \left(\frac{\Lambda V}{\Lambda^{> m} V},\overline{d} \right)\] admits a homotopy retraction.
\end{theorem}

\begin{example}\label{example:FelixHalpNeedSullivan}
	Let $X=S^7$ and $(\Lambda V,d)= (\Lambda (a),0)$, with $|a|=7$, its minimal model. Since $\rho_1=\id$ and $X$ is not contractible, we see that $\cat(X_0)=\cat(X)=1$. Now take a cdga model for $X$ of the form \[A=\left(\frac{\Lambda (a_4,x_3)}{(a^2)},dx=a\right),\] the projection $\rho_1\colon(A,d)\rightarrow \left(\frac{A}{(A^+)^2},d\right)$ is not homology injective since the fundamental class is represented by $ax\in (A^+)^2$. Therefore $\rho_1$ cannot have a homotopy retraction. This shows that, in previous theorem, it is necessary to take a Sullivan model for the space.
\end{example}

One can define new invariants by weakening the requirements on homotopy retractions, namely
\begin{itemize}
	\item the \emph{module LS category} of $X$, $\mcat(X)$, as the smallest $m$ such that $\rho_m$ admits a homotopy retraction as $(\Lambda V,d)$-module,
	\item the \emph{rational Toomer invariant} of $X$, $e(X)$, as the least $m$ such that $\Ho(\rho_m)$ is injective.
\end{itemize}

Hence
\begin{corollary}\label{coro:inequalitiesCat}
	Let $X$ be a space and $(A,d)$ be any cdga model for $X$, then \[\nil\ \Ho^+(X,\mQ)\le e(X)\le\mcat(X)\le\cat(X_0)\le \cat(X)\le \nil\ A^+.\] If $X$ is formal then one can take $(A,d)=(\Ho(A),0)$ and all the above are thus equalities.
\end{corollary}

\begin{proof}
	For the first inequality, suppose $e(X)=m$ and consider \[[x_0],[x_1],\dots,[x_m]\in\Ho^+(X,\mQ)\cong \Ho^+(\Lambda V,d).\] Since one can identify $x_i\in \Lambda^+V$, we have $x_0x_1\cdots x_m\in\Lambda^{>m}V$ and therefore, $[x_0x_1\cdots x_m]\in \ker \Ho(\rho_m)$. Since, by hypothesis, $\Ho(\rho_m)$ is injective, we must have $[x_0][x_1]\cdots[x_m]=[x_0x_1\cdots x_m]=0$. Let us now prove the last inequality. Suppose $\nil\ A^+=m$,  and take a quasi-isomorphism $\theta\colon (\Lambda V,d)\rightarrow (A,d)$. Since $(A^+)^{m+1}=0$, $\theta$ factors through $\rho_m$ as
	\begin{equation*}
	\begin{tikzcd}[row sep=tiny]
	(\Lambda V,d)\ar[dr, twoheadrightarrow, "\rho_m"']\ar[rr, "\theta", "\simeq"']&&(A,d)\\
	&(\frac{\Lambda V}{\Lambda^{>m}V},\overline{d}).\ar[ur, "\overline{\rho}_m"']		
	\end{tikzcd}
	\end{equation*}
	Then taking a relative Sullivan model for $\rho_m$, the lifting lemma gives a homotopy retraction $r$ for $\rho_m$:
	\begin{equation*}
	\begin{tikzcd}
	&(\Lambda V,d)\ar[r,"\id"]\ar[d, hook]\ar[dl, "\rho_m"']&(\Lambda V,d)\ar[d, "\theta", "\simeq"']\\
	(\frac{\Lambda V}{\Lambda^{>m}V},\overline{d})&(\Lambda V\otimes\Lambda Z,D)\ar[l, "\xi", "\simeq"']\ar[r, "\overline{\rho}_m\circ\xi"']\ar[ur, dashed, "r"']&(A,d)
	\end{tikzcd}
	\end{equation*}
		
\end{proof}

\begin{proposition}\label{prop:LScatVodd}
	If $X$ is a space such that $\pi_*(X)\otimes \mQ$ is finite dimensional and concentrated in odd degrees, then \[\cat(X_0)=\dim \pi_*(X)\otimes\mQ.\]
\end{proposition}

\begin{proof}
	By Theorem \ref{th:connectinTopol}, the minimal model of $X$ is of the form $(\Lambda V,d)=(\Lambda( a_1,\ldots,a_n),d)$ with $a_i$ in odd degree and $n=\dim \pi_*(X)\otimes\mQ$. Since $\Lambda^{n+1}V=0$ then $\rho_n=\id$. Now consider the element $\omega:=a_1\cdots a_n\in\Lambda^nV$. For degree reasons, $d(\omega)=0$ and by the nilpotence condition of a Sullivan algebra, $\omega$ cannot be a boundary. This shows that $\rho_n$ has a homotopy retraction and $\rho_{n-1}$ does not as it is not homology injective.
\end{proof}

The following is a first example of how rational homotopy theory gives better lower bounds for sectional category than the standard cohomological ones.

\begin{example} Let $X$ be the space of Example \ref{examp:NotFormalSpace} and recall its minimal model $(\Lambda V,d)=(\Lambda(a,b,x),d)$ with $|a|=|b|=3$ and $d(x)=ab$. By the previous proposition, $\cat(X_0)=3$. However $\nil\ \Ho^+(\Lambda V,d)=2$.
\end{example}

The following is a surprising result by Hess which is a second pillar for the study of rational LS category.

\begin{theorem}[\cite{He91}]\label{th:Hess}
	Let $X$ be a space and $(\Lambda V,d)$ be its minimal model. Then the projection \[\rho_m\colon (\Lambda V,d)\rightarrow \left(\frac{\Lambda V}{\Lambda^{> m} V},\overline{d} \right)\] admits a cdga homotopy retraction if and only if it admits a homotopy retraction as $(\Lambda V,d)$-module. In particular $\mcat(X)=\cat(X_0)$.  
\end{theorem} 

Using this result, F\'elix-Halperin-Lemaire\cite{FHL98} proved

\begin{theorem}\label{th:FelixHalpLemai}
Let $X$ and $Y$ be spaces. Then
\begin{itemize}
	\item $\cat(X_0\times Y_0)=\cat(X_0)+\cat(Y_0)$,
	\item If $X$ is a Poincar\'e duality complex (see Section \ref{ssec:PoincareDuality}), then $e(X)=\cat(X_0).$
\end{itemize}
\end{theorem}

In this theorem, the hypothesis of Poincar\'e duality is necessary:

\begin{example}[\cite{SL81}]
	Let $X$ be the space modeled by the cdga \[(A,d)=\left(\frac{\Lambda(a,b,x)}{(a^4,ab,ax)},d\right),\] with $|a|=2$, $|b|=3$, $d(a)=d(b)=0$ and $d(x)=a^3$. Observe that, as vector spaces, $A=\mQ\langle 1,a,b,a^2,x,a^3,bx\rangle$ and $\Ho(A,d)=\mQ\langle 1,[a],[b],[a]^2,[bx]\rangle$. This shows that $2\le\cat(X_0)\le 3$. The minimal model of $X$ is of the form $(\Lambda V,d)$ with $V=\mQ\langle a,b,v,x,w\rangle\oplus V^{\ge 7}$, $d(a)=d(b)=0$, $d(v)=ab$, $d(x)=a^3$ and $d(w)=bv$. Then $\Ho(\Lambda V,d)=\mQ\langle 1,[a],[b],[a]^2,[bx+a^2v]\rangle$, and there is no element in $\Lambda^{>2}V$ representing a non zero class in $\Ho(\Lambda V,d)$. This shows that $e(X)=2$.\\
	
	We now prove by contradiction that $\cat(X_0)=3$. Suppose that $\rho_2$ admits a cdga homotopy retraction $r$. This implies that $\Ho(r)\circ \Ho(\rho_2)=\id$ but this is absurd since $d(x)\in\Lambda^3 V$ and, for degree reasons,  \[(\Ho(r)\circ \Ho(\rho_2))([bx+a^2v])=H(r)([bx+a^2v])=H(r)([b][x])=H(r)([b])0=0.\]
\end{example}

\subsection{The mapping theorem for LS category}
Another important consequence of Theorem \ref{th:CharacterisationSecat} is the so called mapping theorem \cite[Theorem I]{Felix82}. We include a simple proof based on the proof of \cite[Pg. 389]{Bible} which we will later on extend to Theorem \ref{th:MappingThTC}.

\begin{theorem}\label{th:MappingThCat}
	Let $f\colon Y\to X$ be a continuous map. If $\pi_*(f)\otimes \mQ$ is injective then $\cat(Y_0)\le \cat(X_0)$.
\end{theorem}

\begin{proof}
	Let $(\Lambda V,d)$ and $(\Lambda W,d)$ be the minimal models of $X$ and $Y$ respectively. Then, by Theorem \ref{th:connectinTopol}, the hypothesis tells us that $f$ is modeled by a surjective cdga morphism $\varphi\colon(\Lambda V,d)\to(\Lambda W,d)$. Denote $p\colon PX\rightarrow X$ the based path fibration on $X$, so that $\cat(X)=\secat(p)$. We construct a model of the (homotopy) pullback
	
	\begin{equation*}
	\begin{tikzcd}
	E\ar[r]\ar[d, twoheadrightarrow, "q"']&PX\ar[d, twoheadrightarrow, "p"]\\
	Y\ar[r, "f"']&X
	\end{tikzcd}
	\end{equation*}

(using the acyclic closure of $(\Lambda V,d)$) as the homotopy pushout

\begin{equation*}
\begin{tikzcd}
&(\Lambda V,d)\ar[dl, "\epsilon"']\ar[d, hook]\ar[r, "\varphi"]&(\Lambda W,d)\ar[d, hook, "j"]\\
\mQ&(\Lambda V\otimes \Lambda \hat{V},D)\ar[l, "\simeq"]\ar[r]&(\Lambda W\otimes \Lambda \hat{V}, \overline D)
\end{tikzcd}
\end{equation*}
Here, $j$ is a cdga model for $q$ and $\overline{D}$ verifies $\overline{D}(\hat{V})\subset \Lambda^+W\otimes\Lambda \hat{V}$ and $\overline{D}_0\colon \hat{V}\rightarrow W$ is surjective. Now write $Z:=\ker \overline{D}_0$ and consider the projection $\pi\colon (\Lambda W\otimes \Lambda \hat{V}, \overline D)\rightarrow (\Lambda Z,0)$. Observe that $Q(\pi)\colon(W\oplus\hat{V},\overline{D}_0)\rightarrow (Z,0)$ is a quasi-isomorphism. By Proposition \ref{prop:SullivanAlgs}, $\pi$ is a cdga quasi-isomorphism. Since $\pi\circ j$ is trivial, we deduce that $q_0$ is trivial and thus $\cat(X_0)=\secat(q_0)$. The result follows since rationalization commutes with limits and $\secat(q_0)\le\secat(p_0)$.
\end{proof}	

Observe that in previous proof, the fact that $D(\hat{V})\subset\Lambda^+V\otimes\Lambda\hat{V}$ was crucial in assuring that $\pi$ is a cdga morphism.

\section{The Whitehead and Ganea characterizations}
We start describing the Whitehead and Ganea characterizations for the sectional category of a continuous map $f$ provided that our spaces are CW complexes. We will then translate them to the rational context in the category \textbf{cdga}.

\subsection{The Whitehead characterization}
Take $i\colon A\hookrightarrow X$ a cofibration replacement for $f$. We define the $m$-th \emph{fat wedge} of $i$ as \[T^m(i):=\left\{(x_0,\ldots,x_m)\in X^{m+1}\colon\ \  x_j\in i(A) \mbox{ for at least one } j \right\}\subset X^{m+1}.\] Then, $\secat(f)=\secat(i)$ is the least $m$ for which there exists the dashed map making the following a homotopy commutative diagram

\begin{center}
\begin{tikzcd}
&T^m(i)\arrow[d, hookrightarrow]\\
X\arrow[ur, dashed]\arrow[r, hookrightarrow, "\Delta_{m+1}"']&X^{m+1}.	
\end{tikzcd}
\end{center}
Here $\Delta_{m+1}$ denotes the diagonal map $\Delta_{m+1}(x)=(x,x,\ldots,x)$

\subsection{The Ganea characterization}\label{ssec:GaneaChar}
Take $p\colon E\fib B$ a fibration replacement for $f$. We define the $m$-th \emph{Ganea fibration} for $f$ as \[G_f^m\colon P^m(p)\fib B,\] where \[P^m(p):=\left\{\sum_{j=0}^mt_j e_j\colon\ \  p(e_0)=p(e_j), t_j\ge 0 \mbox{ and } \sum_{j=0}^mt_j=1\right\},\] for which summands of the form $0 e$ are dropped, 
and $G_f^m(t_0 e_0+t_1 e_1+\cdots + t_m e_m):=p(e_0)$. See \cite[Pg. 54]{Husemoller66} for a detailed description of the topology of $P^m(p)$. Then, $\secat(f)=\secat(p)$ is the smallest $m$ for which $G_f^m$ admits a section.

\subsection{Whitehead vs. Ganea}\label{ssec:WhiteVsGanea}
The way to prove the Ganea characterization is to glue the local sections of $p$ into a section of $G_p^m$ by means of a partition of the unity on $B$. Then, to prove the Whitehead characterization it is sufficient to see that there is a homotopy pullback\cite{Ma76,Fasso}

\begin{equation}\label{diag:WhiteVsGane}
	\begin{tikzcd}
		P^m(p)\arrow[r]\arrow[d, "G_p^m"']&T^m(i)\arrow[d, hookrightarrow]\\
		X\arrow[r, hookrightarrow, "\Delta_{m+1}"']&X^{m+1}.	
	\end{tikzcd}
\end{equation}	

\subsection{First algebraic characterizations}\label{ssec:FirstAlgCharac}
Let $\varphi\colon (A,d)\fib (B,b)$ be a surjective cdga model for a cofibration replacement $i\colon A\cofib X$ of a continuous map $f$. Then by \cite[Thm. 1]{Felix09} the inclusion of the $m$-fat wedge $T^m(i)\hookrightarrow X^{m+1}$ is modeled by the quotient \[\pi\colon (A^{\otimes m+1},d)\fib \left(\frac{A^{\otimes m+1}}{(\ker\varphi)^{\otimes m+1}},\overline{d}\right) \] while $\Delta_{m+1}\colon X\cofib X^{m+1}$ is modeled by the multiplication $\mu\colon (A^{\otimes m+1},d)\rightarrow (A,d)$. Now choose a relative Sullivan model for $\pi$, \[j_m\colon (A^{\otimes m+1},d)\cofib (A^{\otimes m+1}\otimes\Lambda W,D),\] then Diagram \ref{diag:WhiteVsGane} is modeled by the following pushout

\begin{equation*}
\begin{tikzcd}
(A^{\otimes m+1},d)\arrow[r, hookrightarrow, "j_m"]\arrow[d, "\mu"']& (A^{\otimes m+1}\otimes\Lambda W,D)\arrow[d, "\overline{\mu}"]\\
(A,d)\arrow[r, hookrightarrow, "i_m"']&(A\otimes\Lambda W,\overline{D}).
\end{tikzcd}
\end{equation*}
Here, $i_m$ is a relative Sullivan model for $G_f^m$. We have that $\secat(f_0)$ is the smallest $m$ for which one of the following equivalent conditions hold

\begin{itemize}
	\item \textbf{Whitehead: } There exists a cdga morphism $r\colon (A^{\otimes m+1}\otimes\Lambda W,D)\rightarrow A$ such that $r\circ j_m=\mu$.
	\item \textbf{Ganea: } There exists a cdga retraction for $i_m$.
\end{itemize}

\section{Rational approximations of sectional category}
One can now impose less restrictive conditions to the existence of morphisms in the characterizations of Section \ref{ssec:FirstAlgCharac} to get algebraic lower bounds for sectional category.\\

Let $f$ be a continuous map and $i_m\colon (A,d)\cofib(A\otimes \Lambda W,d)$ be a relative Sullivan model for $G_f^m$ (for example, the one constructed in Section \ref{ssec:FirstAlgCharac}). If $r\colon (A\otimes \Lambda W,D)\rightarrow (A,d)$ is a retraction of $i_m$ (that is $r(a\otimes 1)=a$, for $a\in A$) then we have the following implications:
$r$ is a cdga morphism $\Rightarrow$ $r$ is an $(A,d)$-module morphism $\Rightarrow$ $\Ho(i)$ is injective.

\begin{definition}
	With previous notation,
	\begin{itemize}
		\item the \emph{module sectional category} of $f$, $\msecat(f)$, is the least $m$ for which $i_m$ admits an $(A,d)$-module retraction.
		
		\item the \emph{homology sectional category} of $f$, $\hsecat(f)$, is the least $m$ for which $\Ho(i_m)=\Ho^*(G_f^m,\mQ)$ is injective. 
	\end{itemize}
\end{definition}

In previous definition one can replace $i_m$ by any model for $G_p^m$ and asking the retractions to be homotopy retractions. We can now deduce

\begin{proposition}
	If $f$ is a continuous map, then \[\nil \ker \Ho^*(f,\mQ)\le \hsecat(f)\le\msecat(f)\le\secat(f_0)\le \secat(f).\]
\end{proposition}

Observe that, if we take $f$ as the inclusion of the base point into a based space $X$, we get
\[\nil\ \Ho^+(X,\mQ)\le e(X) \le\mcat(X)\le\cat(X_0)\le \cat(X).\]

In contrast with Theorem \ref{th:Hess}, there are maps $f$ for which $\msecat(f)\less \secat(f_0)$:

\begin{example}[\cite{St00}]\label{example:Stanley} Consider the cdga $(\Lambda (a_2,b_2,x_3), dx=a^2+b^2)$, then the inclusion \[\varphi\colon(\Lambda (a),0)\cofib (\Lambda (a,b,x), dx=a^2+b^2)\] admits a retraction as $(\Lambda (a),0)$-module but not as cdga. In fact, if $r$ is a cdga retraction then it must verify $r(a)=a$, $r(b)=\alpha a$, $r(x)=0$, but then $r$ cannot commute with differentials since $r(dx)=r(a^2+b^2)=(1+\alpha^2)a^2\ne 0$. Let us now define an $(A,d)$-module retraction as $r(b^ix)=0$ and $r(b^i)=0$ for $i$ odd and $r(b^i)=(-1)^{i/2}a^i$ for $i$ even.
\end{example}

Since we are particularly interested in topological complexity, by taking $f$ as the diagonal map, $\Delta_X^n$, we introduce

\begin{definition} For a space $X$,
	\begin{itemize}
		\item the \emph{module topological complexity} of $X$ is $\mtc_n(X):=\msecat(\Delta_X^n)$,
		\item the \emph{homology topological complexity} is $\htc_n(X):=\hsecat(\Delta_X^n)$.
	\end{itemize}
\end{definition}

Since $(\Delta_X^n)_0=\Delta_{X_0}^n$, we have \[\nil \ker \Ho^*(\Delta_X^n,\mQ)\le \htc_n(X)\le \mtc_n(X)\le \tc_n(X_0)\le \tc_n(X).\]

\subsection{Module sectional category} Most of this section is based on \cite{FGKV06,Carrasquel14b,Carrasquel15b}. Let $(A,d)$ be a cdga. An $(A,d)$-module is a chain complex $(M,d)$ together with an action of $A$, $A\otimes M\rightarrow M$ verifying $d(am)=d(a)m + (-1)^{|a|}ad(m)$. A \emph{semi-free} extension of an $(A,d)$-module $(M,d)$ is an $(A,d)$ module of the form \[(M\oplus(A\otimes X),d),\] where $X=\bigoplus_{i\ge 0}X_i$, $d(X_0)\subset M$, $d(X_k) \subset M\oplus(A\otimes(\bigoplus_{i=0}^{k-1}X_i)$ and the inclusion $(M,d)\rightarrow (M\oplus(A\otimes X),d)$ is an $(A,d)$-module morphism. Observe that relative Sullivan algebras $(A\otimes\Lambda V,D)$ are semi-free extensions of $(A,d)$\cite[Lemma 14.1]{Bible}. In fact, we have that the category of $(A,d)$-modules is a closed proper model category \cite[Theorem 4.1]{FGKV06}.\\

Any cdga morphism $\varphi\colon(A,d)\rightarrow (B,d)$ can be seen as an $(A,d)$-module morphism by taking in $(B,d)$ the $(A,d)$-module structure $ab=\varphi(a)b$. A \emph{semi-free model} for $\varphi$ is an $(A,d)$-module quasi-isomorphism $(A\oplus(A\otimes X),d)\we (B,d)$ from a $(A,d)$-semi-free extension of $(A,d)$. A \emph{semi-free} model for a continuous map $f$ is just a semi-free model for any cdga model of $f$.

\begin{theorem}[\cite{FGKV06}]
	If $(A,d)\cofib(A\oplus(A\otimes X),d)$ is a semi-free model for a map $f$ then a semi-free model for the $m$-th Ganea fibration $G_f^m$ is \[j_m\colon(A,d)\cofib \left(A\oplus (A\otimes s^{-m}X^{\otimes m+1}),d\right)\]
\end{theorem}

Here $d=d_0+d_+$ (in $A \oplus A\otimes s^{-m}X^{\otimes m+1}$) is given by 
\begin{eqnarray*}
	\lefteqn{d(s^{-m}x_0\otimes \cdots \otimes x_m) = (-1)^{\sum \limits_{k=1}^{m}(k|x_{m-k}| + k -1)} d_0x_0\cdot \cdots \cdot d_0x_m}\\
	&+ & \sum \limits _{i=0}^{m} \sum \limits_{j_i}(-1)^{(|a_{ij_i}| +1)(|x_0| + \cdots +|x_{i-1}| + m)}a_{ij_i}\otimes s^{-m}x_0 \otimes \cdots \otimes x_{ij_i}\otimes  \cdots \otimes x_m,
\end{eqnarray*}
for $x_0$,..., $x_m\in X$ and $d_+x_i = \sum \limits_{j_i}a_{ij_i}\otimes x_{ij_i}$ with $a_{ij_i}\in A$ and $x_{ij_i}\in X$.

\begin{corollary}
	If $f$ is a map, then $\msecat(f)$ is the smallest $m$ for which $j_m$ admits an $(A,d)$-module retraction and $\hsecat(f)$ is the smallest $m$ for which $\Ho(j_m)$ is injective.
\end{corollary}

\subsection{Poincar\'e Duality}\label{ssec:PoincareDuality}

A finite dimensional commutative graded algebra $H$ is said to be a \emph{Poincar\'e duality algebra with formal dimension $n$} when $H^0=\mQ$, $H=\bigoplus_{i=o}^n H^i$ and there exists an element $\Omega\in H^n$ such that the map of degree $-n$

\begin{center}
	\begin{tikzcd}[row sep=tiny]
		H\arrow[r, "\Phi"]&\hom(H,\mQ)\\
		a\arrow[r, mapsto]&b\mapsto \Omega^{\#}(ab)\\
	\end{tikzcd}
\end{center}

is an isomorphism, where $\Omega^\#$ denotes the dual of $\Omega$.

\begin{theorem}[\cite{Carrasquel14b}]
	Let $\varphi\colon (A,d)\to (B,d)$ be a cdga morphism with $H(A,d)$ a Poincar\'e duality algebra, then $\Ho(\varphi)$ is injective if and only if $\varphi$ admits a homotopy retraction as a morphism of $(A,d)$-modules.
\end{theorem}

\begin{corollary}
	If $f\colon X\rightarrow Y$ is a map with $\Ho^*(Y,\mQ)$ a Poincar\'e duality algebra, then $\msecat(f)=\hsecat(f)$. In particular, if $X$ is a Poincar\'e duality complex, then $\mtc_n(X)=\htc_n(X)$.
\end{corollary}

\section{Characterization \`a la F\'elix-Halperin}
Let $f$ be a continuous map and $\varphi\colon(A,d)\fib(B,d)$ a surjective cdga model for $f$. Recall the notation from  Section \ref{ssec:FirstAlgCharac}. Then, since $\mu((\ker\varphi)^{\otimes m+1})\subset (\ker \varphi)^{m+1}$, we have a diagram
\begin{equation*}
\begin{tikzcd}
	&\left(\frac{A^{\otimes m+1}}{(\ker \varphi)^{\otimes m+1}},\overline{d}\right)\ar[dddr, bend left]\\
	(A^{\otimes m+1},d)\arrow[ur, twoheadrightarrow, "\pi"]\arrow[r, hookrightarrow, "j"]\arrow[d, "\mu"']& (A^{\otimes m+1}\otimes\Lambda W,D)\arrow[u, "\simeq"]\arrow[d, "\overline{\mu}"]\\
	(A,d)\ar[drr, bend right, "\rho_m"']\arrow[r, hookrightarrow, "i"']&(A\otimes\Lambda W,\overline{D})\ar[dr, dashed, "\tau"]\\
	&&\left(\frac{A}{(\ker\varphi)^{m+1}},\overline{d}\right).
\end{tikzcd}
\end{equation*}
Then taking $(A\otimes \Lambda Z,D)$ a relative Sullivan model for $\rho_m$, the lifting lemma gives a diagram

\begin{equation*}
\begin{tikzcd}
A\arrow[r, hook]\arrow[d, hook, "i"]&(A\otimes\Lambda Z, D)\arrow[d, twoheadrightarrow, "\simeq"]\\
(A\otimes\Lambda W,\overline{D})\ar[ur, "\overline{\tau}"]\ar[r, "\tau"']&\frac{A}{(\ker\varphi)^{m+1}}
\end{tikzcd}
\end{equation*}
which implies

\begin{proposition}[\cite{Jessup12,Carrasquel10}]\label{prop:upperBoudnSecat}
	If $\rho_m$ admits a cdga homotopy retraction then $\secat(f_0)\le m$. If $\rho_m$ admits an $(A,d)$-module homotopy retraction, then $\msecat(f)\le m$. If $\Ho(\rho_m)$ is injective, then $\hsecat(f)\le m$.
\end{proposition}

Observe that the opposite implication need not hold (see Example \ref{example:FelixHalpNeedSullivan}) and this remains true even for topological complexity:

\begin{example}
	Let us compute the topological complexity of $S^3_0$. The path fibration $\pi_2$ is modeled by the multiplication morphism $\mu_2\colon(\Lambda(a_1,a_2),0)\rightarrow (\Lambda a,0)$ with $|a_1|=|a_2|=|a|=3$. We remark that $\ker \Ho(\mu_2)=([a_1]-[a_2])$, therefore $\nil \ker \Ho(\mu_2)=1\le \tc(S^3_0)$. Now, since $\ker\mu_2=(a_1-a_2)$, we have that $(\ker\mu_2)^2=0$ and $\rho_1=\id$ admits a (homotopy) retraction. This proves, by Proposition \ref{prop:upperBoudnSecat}, that $\tc(S^3_0)=1$. Another way to show this is through $\tc(S^3_0)\le\tc(S^3)=1$.\\
	
	Now consider another cdga model for $S^3$, namely the cdga $(A,d):=(\Lambda (a,b,c),d)$ with $|a|=|b|=|c|=1$, with $d(a)=bc$, $d(b)=ac$ and $d(c)=ab$. Observe that $A$ is a free cga but that $(A,d)$ is \emph{not} a Sullivan algebra. Observe also that $(A,d)$ is a cdga model for $S^3$ as there is a quasi-isomorphism $(\Lambda v_3,0)\we (A,d)$ defined as $v\mapsto abc$.\\
	
	We have then that another cdga model for the path fibration $\pi_2$ is the multiplication on $(A,d)$, $\mu_2\colon(\Lambda (a_1,b_1,c_1,a_2,b_2,c_2),d)=(\Lambda(a,b,c),d)$. We now see that \[\rho_2\colon (\Lambda(a_1,b_1,c_1,a_2,b_2,c_2),d)\rightarrow \left(\frac{\Lambda(a_1,b_1,c_1,a_2,b_2,c_2)}{(\ker \mu_2)^3},d\right)\] is not homology injective. In fact, the element $\omega:=(a_1-a_2)(b_1-b_2)(c_1-c_2)\in(\ker\mu_2)^3$ can be written as $\omega=a_1b_1c_1-a_2b_2c_2-d(a_1a_2-b_1b_2+c_1c_2)$. This means that $[\omega]\ne 0$ and that $\Ho(\rho_2)([\omega])=0$.
\end{example}

If $I$ is an ideal of $A$, define the \emph{homology nilpotency} of $I$, $\hnil\ I$, to be the smallest $m$ such that $I^{m+1}$ is contained in an acyclic ideal of $A$, that is, a differential ideal $J$, with $\Ho(J)=0$.\\

Using previous proposition one can deduce that, if $\varphi$ is a surjective cdga model for $f$, then \[\secat(f_0)\le \hnil (\ker\varphi)\le \nil\ \ker\ \varphi.\]

In fact, if $(\ker \varphi)^{m+1}\subset J$ with $J$ an acyclic ideal of $(A,d)$. Then, by the five lemma, there is a diagram

\begin{equation*}
\begin{tikzcd}
&(A,d)\arrow[dl, twoheadrightarrow, "\simeq"']\ar[dr, twoheadrightarrow, "\rho_m"]&\\
\left(\frac{A}{J},\overline{d}\right)&&\left(\frac{A}{(\ker\varphi)^{m+1}},\overline{d}\right)\ar[ll]
\end{tikzcd}
\end{equation*}
which can be used, together with the lifting lemma, to give a homotopy retraction for $\rho_m$.\\

As a consequence, if $f$ is a formal map such that $\Ho^*(f,\mQ)$ is surjective, then $\secat(f_0)=\nil \ker \Ho^*(f,\mQ)$. The hypothesis of $\Ho^*(f,\mQ)$ being surjective is necessary:

\begin{example}
Consider the cdga morphism $\varphi\colon (\Lambda a_2,0)\rightarrow \left(\frac{\Lambda (a_2,b_2)}{(a^2+b^2)},0\right)$ defined by $\varphi(a)=a$. Let $f$ be a continuous map whose model is $\varphi$ (the spatial realization of $\varphi$, for instance), then $\Ho^*(f,\mQ)=\varphi$. Obviously $\nil\ker\varphi=0$ but $\varphi$ does not admit a homotopy retraction. In fact, a relative Sullivan model for $\varphi$ is given by

\begin{equation*}
\begin{tikzcd}
(\Lambda (a),0)\ar[r, "\varphi"]\ar[dr, hook, "j"']&\left(\frac{\Lambda (a,b)}{(a^2+b^2)},0\right)\\
&(\Lambda (a)\otimes\Lambda (b,v_1,v_2,\ldots), D),\ar[u, "\simeq", "\theta"']
\end{tikzcd}
\end{equation*}
with $D(b)=0$, $D(v_1)=a^2+b^2$, $\theta(b)=b$, $\theta(v_1)=0$, and $|v_i|\ge 4$ for $i\ge 2$. But, as in Example \ref{example:Stanley}, $j$ cannot have a cdga retraction. This means that $\secat(f_0)\ge 1$.
\end{example}


\subsection{When $f$ admits a homotopy retraction}

Now suppose that our continuous map $f$ admits a homotopy retraction. Observe that this is the case of the path fibrations $\pi_n\colon X^I\fib X^n$. We now construct a special type of model for such maps which we will call s-models.\\

We can suppose that $f\colon X\cofib Y$ is a cofibration and that it admits a strict retraction $r\colon Y\rightarrow X$, so that $r\circ f=\id_X$. Now, take a surjective quasi-isomorphism $\theta\colon (\Lambda V,d)\rightarrow \apl(X)$ and a relative Sullivan model

\begin{equation*}
\begin{tikzcd}
\apl(X)\ar[r, "\apl(r)"]&\apl(Y)\\
(\Lambda V,d)\ar[u, twoheadrightarrow,"\simeq"',"\theta"]\ar[r, hook]&(\Lambda V\otimes \Lambda W,D).\ar[u, "\simeq", "\xi"']
\end{tikzcd}
\end{equation*}
Now the lifting lemma
\begin{equation*}
\begin{tikzcd}
(\Lambda V,d)\ar[rr, "\id"]\ar[d, hookrightarrow]&&(\Lambda V,d)\ar[d, twoheadrightarrow, "\simeq"', "\theta"]\\
(\Lambda V\otimes\Lambda W, D)\ar[urr, dashed, "\varphi'"]\ar[r, "\xi"']&\apl(Y)\ar[r, "\apl(f)"']&\apl(X),
\end{tikzcd}
\end{equation*}
gives a cdga morphism, $\varphi'$, which is a model for $f$. Now suppose $(A,d)$ is any cdga model for $X$. Then there is a pushout diagram
\begin{equation*}
\begin{tikzcd}
(\Lambda V,d\ar[d, "\psi"', "\simeq"])\ar[r, hookrightarrow]&(\Lambda V\otimes\Lambda W, D)\ar[d, "\simeq"]\ar[ddr, bend left, "\psi\circ\varphi'"]&\\
(A,d)\ar[drr, bend right, "\id_A"]\ar[r, hookrightarrow]&(A\otimes \Lambda W, \overline{D})\ar[dr, dashed, "\varphi"]&\\
&&(A,d),
\end{tikzcd}
\end{equation*}
which gives a model $\varphi$ for $f$ which is a retraction for the inclusion \[(A,d)\cofib (A\otimes\Lambda W,\overline{D}).\] Such a model is said to be an \emph{s-model} for $f$.\\

We can now state a generalization of the F\'elix-Halperin theorem which lets us compute the topological complexity of rational spaces.

\begin{theorem}[\cite{Carrasquel14c}]\label{th:CharacterisationSecat} Let $f$ be a continuous map and $\varphi\colon (A\otimes\Lambda W,D)\rightarrow (A,d)$ an s-model for $f$. Then $\secat(f_0)$ is the smallest $m$ for which the cdga projection \[\rho_m\colon (A\otimes \Lambda W, D)\rightarrow \left(\frac{A\otimes \Lambda W}{(\ker\varphi)^{m+1}}, \overline{D}\right)\] admits a homotopy retraction. Also,
\begin{itemize}
	\item $\msecat(f)$ is the smallest $m$ for which $\rho_m$ admits a homotopy retraction as $(A\otimes\Lambda W,D)$-module, and
	
	\item $\hsecat(f)$ is the smallest $m$ for which $\Ho(\rho_m)$ is injective.
\end{itemize}
\end{theorem}

Observe that the augmentation is an s-model for the base point inclusion. In this case, previous theorem is just the F\'elix-Halperin Theorem \ref{th:FelixHalperinCharac}.

\subsection{Applications to topological complexity}
Let $(A,d)$ be any cdga model for a space $X$. We build an s-model for the path fibration $\pi_n\colon X^I\rightarrow X^n$ as follows: Take $(\Lambda V,d)$ a Sullivan model for $X$ (not necessarily minimal) and take a quasi-isomorphism $\theta\colon(\Lambda V,d)\we (A,d)$. Then the cdga morphism \[\mu^\theta_n:=(\id_A,\theta,\ldots,\theta)\colon (A,d)\otimes(\Lambda V,d)^{\otimes n-1}\rightarrow (A,d)\] is an s-model for $\pi_n$. Moreover, as with the multiplication ($\mu_n=\mu_n^\id$), one can see that $\ker\mu_n^\theta=(\theta(v)- v_i)$ where, as usual, $v_i$, $i=2,\ldots,n$ stands for $v$ in the $i$-th factor. As a corollary to Theorem \ref{th:CharacterisationSecat}, we get

\begin{corollary}\label{cor:characterizationTC}
	Let $X$ be a space, $(\Lambda V,d)$ a Sullivan model for $X$, and \[\theta\colon(\Lambda V,d)\we (A,d)\] a quasi-isomorphism. Then $\tc_n(X_0)$ is the least $m$ such that the projection \[\rho_m\colon (A\otimes(\Lambda V)^{\otimes n-1},d)\rightarrow \left(\frac{A\otimes(\Lambda V)^{\otimes n-1}}{(\ker\mu_n^\theta)^{m+1}},\overline{d}\right)\] admits a homotopy retraction. Moreover, $\mtc_n(X)$ is the least $m$ such that $\rho_m$ admits a retraction as $(A\otimes(\Lambda V)^{\otimes n-1},d)$-module and $\htc_n(X)$ is the least $m$ such that $\Ho(\rho_m)$ is injective.
\end{corollary}

\begin{example} Compare \cite[Example 6.5]{FGKV06} and \cite[Example 3.5]{Grant}. Let $X=S_a^3\vee S_b^3\cup e^8\cup e^8$ where the cells are attached through the iterated Whitehead products $[S_a^3,[S_a^3,S_b^3]]$ and $[S_b^3,[S_a^3,S_b^3]]$. By \cite[p. 179]{Bible}, the minimal model of $X$ is of the form $(\Lambda V,d)$ with $V^{\le 10}=\mQ\langle a,b,x,y\rangle$, $|a|=|b|=3$, $d(a)=d(b)=0$, $d(x)=ab$, $d(y)=abx$ and $\Ho(\Lambda V,d)=\mQ\langle1,[a],[b],[ax],[bx]\rangle$. Therefore, writing  \[(A,d):= \left(\frac{\Lambda(a,b,x)}{(abx)},d\right),\] the projection $\theta\colon (\Lambda V,d)\rightarrow(A,d)$ is a quasi-isomorphism and an s-model for the diagonal map $\Delta_X$ is 
\[\mu^\theta\colon\left(\frac{\Lambda(a_1,b_1,x_1)}{(abx)}\otimes\Lambda\left( \langle a_2,b_2,x_2,y_2\rangle\oplus V^{\ge 11}\right),d\right)\rightarrow \left(\frac{\Lambda(a,b,x)}{(abx)},d\right).\] Now consider the element $\omega:=(a_1-a_2)(b_1-b_2)(x_1-x_2)\in (\ker \mu^\theta)^{3}$. A computation yields that $\omega=a_1x_1b_2-b_1a_2x_2+a_1b_2x_2-b_1x_1a_2-d(x_1x_2+y_2)$.	This means that $[\omega]\in \ker \Ho(\rho_2)$ is non-zero. Therefore $\htc(X)\ge 3$. On the other hand, the $\nil \ker \mu=3$ where $\mu\colon A\otimes A\rightarrow A$ is the multiplication. This proves that $3=\tc(X_0)\le \tc(X)$ whilst $\nil \ker \Ho^*(\Delta_X^2,\mQ)=2$.
\end{example}

Finally we have

\begin{theorem}[\cite{Carrasquel15b}] If $X$ and $Y$ are topological spaces, then \[\mtc_n(X\times Y)=\mtc_n(X)+\mtc_n(Y).\]
	
\end{theorem}

\section{A mapping theorem for topological complexity}
In this section we will give a slight generalization of the mapping theorem for rational topological complexity of Grant-Lupton-Oprea\cite[Theorem 3.2]{Grant} and give a proof using Sullivan models.

\begin{theorem}\label{th:MappingThTC}
	Let $f_i\colon Y_i\rightarrow X$ be continuous maps, $i=1,\ldots,n$, between rational spaces such that $\pi_*(f_i)$ are injective. If $\im\ (\pi_*(f_1))\cap \im\ (\pi_*(f_2))=0$ then \[\cat(Y_1)+\cdots+\cat(Y_n)\le\tc_n(X).\]
\end{theorem}

\begin{proof}
Since $\pi_*(f_i)$ is injective, there is a surjective model for $f_i$, \[\varphi_i\colon (\Lambda V,d)\fib(\Lambda W_i,d)\] between minimal Sullivan models. The condition $\im\ (\pi_*(f_1))\cap \im\ (\pi_*(f_2))=0$ implies that for each $w_1\in W_1$ and $w_2\in W_2$, there exist $v_1,v_2\in V$ such that $\varphi_1(v_1)=w_1$, $\varphi_2(v_1)=0$, $\varphi_1(v_2)=0$ and $\varphi_2(v_2)=w_2$. As in the proof of Theorem \ref{th:MappingThCat}, we model the (homotopy) pullback
\begin{equation*}
\begin{tikzcd}[column sep=large]
P\ar[r]\ar[d, twoheadrightarrow, "q"']&X^I\ar[d, twoheadrightarrow, "\pi_n"]\\
Y_1\times\cdots \times Y_n\ar[r, "f_1\times\cdots\times f_n"']&X^n,
\end{tikzcd}
\end{equation*}
through the pushout
\begin{equation*}
\begin{tikzcd}
(\Lambda V_1\otimes\cdots\otimes\Lambda V_n,d)\ar[d, hook, "i"']\ar[r, twoheadrightarrow, "\varphi_1\otimes\cdots\otimes\varphi_n"]&(\Lambda W_1\otimes\cdots\otimes\Lambda W _n,d)\ar[d, hook, "j"]\\
(\Lambda V_1\otimes\cdots\otimes\Lambda V_n\otimes\Lambda \hat{V},D)\ar[r, twoheadrightarrow]&(\Lambda W_1\otimes\cdots\otimes\Lambda W_n\otimes\Lambda \hat{V},\overline{D}),
\end{tikzcd}
\end{equation*}
where $i$ is the relative model for $\pi_n$ from Section \ref{sec:modelForDiagonal}, $j$ is a model for $q$, $\hat{V}=\hat{V_1}\oplus\cdots\oplus\hat{V}_{n-1}$, and $\overline{D}_0(\hat{v}_i)=\varphi_{i+1}(v_{i+1})-\varphi_i(v_i)$. We will prove that $q$ is trivial. Observe that, by the properties of the $\varphi_i$'s, $\overline{D}_0\colon \hat{V}\rightarrow W_1\oplus\cdots\oplus W_n$ is surjective. Now take $Z:=\ker \overline{D}_0\subset \hat{V}$. Since $\overline{D}(\hat{V})\subset \Lambda^+(W_1\oplus\cdots\oplus W_n)\otimes\Lambda \hat{V}$, the projection \[\xi\colon (\Lambda W_1\otimes\cdots\otimes\Lambda W_n\otimes\Lambda \hat{V},\overline{D})\fib (\Lambda Z,0)\] is well defined. By construction, $\Ho(Q(\xi))\colon\Ho(W_1\oplus\cdots W_n\oplus\hat{V},\overline{D}_0)\rightarrow \Ho(Z,0)$ is an isomorphism, therefore, by Proposition \ref{prop:SullivanAlgs}, $\xi$ is a quasi isomorphism. Since $\xi\circ j$ is trivial we have that $q$ is trivial, thus \[\cat(Y_1\times\cdots\times Y_n)=\secat(q)\le\secat(\pi_n)=\tc_n(X).\]

But by Theorem \ref{th:FelixHalpLemai}, $\cat(Y_1\times\cdots\times Y_n)=\cat(Y_1)+\cdots+\cat(Y_n)$.
\end{proof}

\section*{Acknowledgements} 
The author thanks the organizers of \textit{Workshop on TC and Related Topics}, held in Oberwolfach in March 2016 and acknowledges the Belgian Interuniversity Attraction Pole (IAP) for support within the framework ``Dynamics, Geometry and Statistical Physics'' (DYGEST).

\providecommand{\bysame}{\leavevmode\hbox to3em{\hrulefill}\thinspace}
\providecommand{\MR}{\relax\ifhmode\unskip\space\fi MR }
\providecommand{\MRhref}[2]{%
	\href{http://www.ams.org/mathscinet-getitem?mr=#1}{#2}
}
\providecommand{\href}[2]{#2}


\begin{thebibliography}{{Car}16b}
	
	\bibitem[BG76]{Bousfield76}
	A.~K. Bousfield and V.~K. A.~M. Gugenheim, \emph{On {${\rm PL}$} de {R}ham
		theory and rational homotopy type}, Mem. Amer. Math. Soc. \textbf{8} (1976),
	no.~179, ix+94. \MR{0425956 (54 \#13906)}
	
	\bibitem[{Car}15]{Carrasquel10}
	J.G. {Carrasquel-Vera}, \emph{{Computations in rational sectional category}},
	Bull. Belg. Math. Soc. Simon Stevin \textbf{22} (2015), no.~3, 455--469.
	
	\bibitem[{Car}16a]{Carrasquel15a}
	\bysame, \emph{The {G}anea conjecture for rational approximations of sectional
		category}, J. Pure Appl. Algebra \textbf{220} (2016), no.~4, 1310--1315,
	http://dx.doi.org/10.1016/j.jpaa.2015.09.001.
	
	\bibitem[{Car}16b]{Carrasquel14c}
	\bysame, \emph{The rational sectional category of certain maps}, Ann. Scuola
	Norm-Sci (2016), To appear. arXiv:1503.07314.
	
	\bibitem[CKV16]{Carrasquel14b}
	J.G. {Carrasquel-Vera}, T.~{Kahl}, and L.~{Vandembroucq}, \emph{{Rational
			approximations of sectional category and {P}oincar\'e duality}}, Proc. Am.
	Math. Soc. \textbf{144} (2016), 909--915,
	\url{http://dx.doi.org/10.1090/proc12722}.
	
	\bibitem[CPV16]{Carrasquel15b}
	J.G. {Carrasquel-Vera}, P.-E. Parent, and L.~Vandembroucq, \emph{Module
		sectional category of products}, Preprint, arXiv:1601.06575.
	
	\bibitem[Far03]{Farber03}
	M.~Farber, \emph{Topological complexity of motion planning}, Discrete Comput.
	Geom. \textbf{29} (2003), no.~2, 211--221. \MR{1957228 (2004c:68132)}
	
	\bibitem[{Fas}02]{Fasso}
	A.~{Fass{\`o} Velenik}, \emph{Relative homotopy invariants of the type of the
		{L}usternik-{S}chnirelmann category}, Eingereichte Dissertation (Ph. D.
	Thesis), Freien Universit{\"a}t Berlin (2002).
	
	\bibitem[FGKV06]{FGKV06}
	L.~{Fern{\'a}ndez~Su{\'a}rez}, P.~Ghienne, T.~Kahl, and L.~Vandembroucq,
	\emph{Joins of {DGA} modules and sectional category}, Algebr. Geom. Topol.
	\textbf{6} (2006), 119--144. \MR{2199456 (2006k:55011)}
	
	\bibitem[FH82]{Felix82}
	Y.~F{\'e}lix and S.~Halperin, \emph{Rational {LS} category and its
		applications}, Trans. Amer. Math. Soc. \textbf{273} (1982), no.~1, 1--38.
	\MR{664027 (84h:55011)}
	
	\bibitem[FHL98]{FHL98}
	Y.~F{\'e}lix, S.~Halperin, and J.-M. Lemaire, \emph{The rational {LS} category
		of products and of {P}oincar\'e duality complexes}, Topology \textbf{37}
	(1998), no.~4, 749--756. \MR{1607732 (99a:55003)}
	
	\bibitem[FHT01]{Bible}
	Y.~F{\'e}lix, S.~Halperin, and J.-C. Thomas, \emph{Rational homotopy theory},
	Graduate Texts in Mathematics, vol. 205, Springer-Verlag, New York, 2001.
	\MR{1802847 (2002d:55014)}
	
	\bibitem[FT88]{FT88}
	Y.~F{\'e}lix and D.~Tanr{\'e}, \emph{Formalit\'e d'une application et suite
		spectrale d'{E}ilenberg-{M}oore}, 99--123. \MR{952575 (89i:55007)}
	
	\bibitem[FT09]{Felix09}
	\bysame, \emph{Rational homotopy of the polyhedral product functor}, Proc.
	Amer. Math. Soc. \textbf{137} (2009), no.~3, 891--898. \MR{2457428
		(2009i:55011)}
	
	\bibitem[GLO15]{Grant}
	M.~{Grant}, G.~{Lupton}, and J.~{Oprea}, \emph{A mapping theorem for
		topological complexity}, Algebr. Geom. Topol. \textbf{15} (2015), no.~3,
	1643--1666.
	
	\bibitem[Hal83]{Ha83}
	S.~Halperin, \emph{Lectures on minimal models}, M\'em. Soc. Math. France (N.S.)
	(1983), no.~9-10, 261. \MR{736299 (85i:55009)}
	
	\bibitem[Hes91]{He91}
	K.~Hess, \emph{A proof of {G}anea's conjecture for rational spaces}, Topology
	\textbf{30} (1991), no.~2, 205--214. \MR{1098914 (92d:55012)}
	
	\bibitem[Hus94]{Husemoller66}
	D.~Husemoller, \emph{Fibre bundles}, third ed., Graduate Texts in Mathematics,
	vol.~20, Springer-Verlag, New York, 1994. \MR{1249482 (94k:55001)}
	
	\bibitem[JMP12]{Jessup12}
	B.~Jessup, A.~Murillo, and P.-E. Parent, \emph{Rational topological
		complexity}, Algebr. Geom. Topol. \textbf{12} (2012), no.~3, 1789--1801.
	\MR{2979997}
	
	\bibitem[LS34]{Lusternik34}
	L.~Lusternik and L.~Schnirelmann, \emph{M{\'e}thodes topologiques dans les
		probl{\`e}mes variationnels}, vol. 188, Hermann, Paris, 1934.
	
	\bibitem[LS81]{SL81}
	J.-M. Lemaire and F.~Sigrist, \emph{Sur les invariants d'homotopie rationnelle
		li\'es \`a la l.s. cat\'egorie.}, Commentarii mathematici Helvetici
	\textbf{56} (1981), 103--122.
	
	\bibitem[Mat76]{Ma76}
	M.~Mather, \emph{Pull-backs in homotopy theory}, Canad. J. Math. \textbf{28}
	(1976), no.~2, 225--263. \MR{0402694 (53 \#6510)}
	
	\bibitem[Opr86]{Oprea86}
	J.~Oprea, \emph{{DGA} homology decompositions and a condition for formality},
	Illinois J. Math. \textbf{30} (1986), no.~1, 122--137.
	
	\bibitem[Qui67]{Quillen67}
	D.~Quillen, \emph{Homotopical algebra}, Lecture Notes in Mathematics, No. 43,
	Springer-Verlag, Berlin-New York, 1967. \MR{0223432 (36 \#6480)}
	
	\bibitem[Rud10]{Rudyak10}
	Y.~Rudyak, \emph{On higher analogs of topological complexity}, Topology Appl.
	\textbf{157} (2010), no.~5, 916--920.
	
	\bibitem[Sch66]{Schwarz66}
	A.~Schwarz, \emph{The genus of a fiber space}, A.M.S Transl. \textbf{55}
	(1966), 49--140.
	
	\bibitem[Sta00]{St00}
	D.~Stanley, \emph{The sectional category of spherical fibrations}, Proc. Am.
	Math. Soc. \textbf{128} (2000), no.~10, 3137--3143.
	
	\bibitem[Sul77]{Su77}
	D.~Sullivan, \emph{Infinitesimal computations in topology}, Inst. Hautes
	\'Etudes Sci. Publ. Math. (1977), no.~47, 269--331. \MR{0646078 (58 \#31119)}
	
	\bibitem[{Vig}79]{Vig79}
	M.~{Vigu\'e-Poirrier}, \emph{Formalit\'e d'une application continue}, C.R.
	Acad. Sci. Paris Sr. A-B (1979), no.~289, 809--812.
	
\end{thebibliography}
\end{document}